\documentclass[final,authoryear]{elsarticle}

\usepackage{lineno,hyperref}
\modulolinenumbers[5]

\usepackage{mathtools}

\usepackage{todonotes}

\newcommand{\supp}[1]{\ensuremath{\text{Support}{(#1)}}}

\usepackage{natbib}
\usepackage{amsmath,mathtools}
\usepackage{dsfont}
\usepackage{amssymb}
\usepackage{xcolor}

\usepackage{graphics, tikz}
\usetikzlibrary{arrows,automata}
\usepackage[latin1]{inputenc}
\usepackage{pgf}

\usepackage{todonotes}

\newtheorem{theorem}{Theorem}
\newtheorem{lemma}[theorem]{Lemma}
\newtheorem{corollary}[theorem]{Corollary}

\newproof{proof}{Proof}
\newtheorem{remark}{Remark}
\newtheorem{definition}{Definition}
\newtheorem{example}{Example}
\newcommand{\comm}[1]{}

\numberwithin{equation}{section}

\newcommand{\ignore}[1]{ }

\begin{document}

\begin{frontmatter}

\title{Evolutionary Stability of Polymorphic Population States in Continuous Games}

\author[DH]{Dharini Hingu\corref{corresponding author}}
\cortext[corresponding author]{Corresponding author}
\address[DH]{Department of Mathematics, Indian Institute of Technology Madras, Chennai 600036, India.}
\ead{dharinihingu@gmail.com}

\author[KSMR]{K.S. Mallikarjuna Rao}
\ead{mallik.rao@iitb.ac.in}
\address[KSMR]{Industrial Engineering and Operations Research,
Indian Institute of Technology Bombay, Mumbai 400076, India.}

\author[AJS]{A.J. Shaiju}
\address[AJS]{Department of Mathematics, Indian Institute of Technology Madras, Chennai 600036, India.}
\ead{ajshaiju@iitm.ac.in}

\begin{abstract}
In games with continuous strategy spaces, if a rest point of the replicator dynamics is asymptotically stable then the
rest point must be finitely supported (\cite{VeelenSpreij}). In this article, we address the converse question that is, we prove that a finitely supported population state is asymptotically stable with respect to the variational norm when it is strongly uninvadable.
\end{abstract}


\end{frontmatter}

\linenumbers

\section{Introduction}

The growing interest in evolutionary games with continuous strategy spaces is primarily because of
the fact that many applications in  economics are modeled as evolutionary games with continuous strategy spaces.
Some of the important  applications include oligopoly games, bargaining games, harvest preemptive game, and war of attrition.
In these games, the strategy space is typically a compact subset of an Euclidean space.  Considering the vast literature on
evolutionary games  with finite strategy spaces (see for e.g., \cite{MaS,MSmithBook,Weibull,HofSig,Cressman,Sandholm}),
a natural approach is to approximate the strategy space by
finite sets.

\cite{OR1} have already considered such an approach for the Harvest preemption game.
It is observed  (see  Section 5.4 in \cite{OR1}) that
the limit of the ESS of the finite approximation is not ESS of the harvest preemption game.  The reason
for such a negative result is the infinite dimensional nature of the space of all mixed strategy spaces.

Evolutionary games with continuous strategy spaces were first studied  by Bomze and P\"{o}tscher through what they called as ``generalized'' mixed strategy games (\cite{Bomze}). There is a vast literature relating evolutionary stability and dynamic stability under various dynamics for games with finite strategy spaces (see for e.g., \cite{Weibull,Cressman,Sandholm,HofSig1,BomzeWeibull}).  However the literature in the continuous games is very little. In the works
\cite{Bomze,Bom,Bom1,OR1,OR2,Cressman1}, the relation is explored for replicator dynamics.
For the literature concerning other dynamics,  see \cite{Cheung2014,FriedOstrov2013,HofOechRiedel2009,LahkarRiedel2015}.

Replicator dynamics is one of the most important dynamics. In this article we study the relation between evolutionary
stability and dynamic stability under replicator dynamics. In \cite{Bomze,OR1,OR2}, such a connection has been studied for a monomorphic population states. In this article, we explore the relation for polymorphic population states.

The population states in evolutionary games with continuous strategy spaces are nothing but probability measures on its strategy space. We can define the static stability concepts of evolutionarily stable strategy (ESS) and uninvadability in these games. We can also define the replicator dynamics for these games which captures the evolution of the population over time. Certain other stability concepts can also be defined, but they depend on the notion of ``closeness" of the population states.

The closeness of the population states can be made precise using various metrics. In this article, we study the evolutionary dynamics and related stability results under the metric defined by the variational or strong norm, which gives rise to the strong topology. In particular, we can define the static stability concepts of strong uninvadability (\cite{Bom,Bom1}) and strong unbeatability (\cite{BomzeWeibull}) along with the dynamic stability concepts such as Lyapunov stability and strongly attracting population states (\cite{Bomze,OR1}). We can also define the concepts of evolutionary robustness (\cite{OR2}) and weakly attracting population states (\cite{OR1}) when we consider the metric associated with the weak topology.

Evolutionary games with continuous strategy spaces and the underlying topology as the strong topology are studied in \cite{OR1}. They connected the static and dynamic stabilities for a population state where all the individuals in the entire population play one and the same pure strategy $x$. Such a population state is called a monomorphic population state and it is represented by the Dirac measure $\delta_x$. They proved that an uninvadable monomorphic population state is Lyapunov stable. Moreover, if the initial population state is close to this monomorphic population state in the strong sense and the payoff function is continuous, then the monomorphic population state is weakly attracting.

In \cite{OR2}, the authors proved that for a doubly symmetric game with continuous payoff function and compact strategy space, an evolutionary robust population state is Lyapunov stable when the underlying topology is the weak topology. They also prove the asymptotic stability of monomorphic evolutionary robust strategies under various assumptions. Cressman also studied evolutionary games with continuous strategy games in \cite{Cressman1}. One of the main result he proves in this paper is regarding population states which have a finite support, which are also known as polymorphic population states. He proved that a dimorphic (or polymorphic) neighbourhood superior population state is neighbourhood attracting.

Bomze has proved a couple of results regarding the asymptotic stability of population states in \cite{Bom}. In the first result \cite[Theorem 2]{Bom}, one of the condition that he imposes is that the set off all population states has to be compact under the given topology. When the underlying topology is strong, this becomes a very strong imposition. We can weaken the conditions of this theorem considerably when we talk about the stability of polymorphic population states. In fact, in this article we prove that for the Lyapunov stability of the polymorphic population state, the strong unbeatability condition is sufficient.

Another result that Bomze (Theorem 3, \cite{Bom}) gives, depends on a first-order condition for strong uninvadability of a population state. This condition may not be satisfied always when the population state is strongly uninvadable. We present here with an example (Example \ref{counter example} in Section 3) where this first-order condition is not satisfied even though the population state is strongly uninvadable. In this article we also prove that in the case of polymorphic population states, strong uninvadability is a sufficient condition for asymptotic stability which is our main result (Theorem \ref{Asymptotic Stability of P^*} in Section 3).

In all the results mentioned above, to obtain dynamic stability for a population state, say $P$, the initial population state, say $Q(0)$, for the replicator dynamics is taken from a small neighbourhood of $P$. Moreover, $Q(0)$ is chosen such that its support contains the support of $P$. This is a necessary condition to study stability with respect to replicator dynamics since the replicator dynamics can only increase or decrease the frequency of the strategies which already exists at the start of the dynamics. Thus, in general, population state $P$ will be not be stable with respect to its complete neighbourhood.

In \cite{VeelenSpreij}, the authors prove that when a population state $P$ is asymptotically stable with respect to a complete (strong) neighbourhood under the replicator dynamics then $P$ should be a polymorphic population state. The results we prove in this article establish the converse of the result by van Veelen and Spreij.

The rest of the article is structured  as follows. Section 2 gives the preliminary notations, definitions and results to study continuous strategy evolutionary games with the underlying topology as the strong topology. Section 3 is divided  into two parts. In the first part we discuss some properties of polymorphic population states. In the second part, we discuss the stability of polymorphic population states. We prove that a strongly uninvadable polymorphic population state is asymptotically stable whereas for Lyapunov stability, it is enough if the population state is strongly unbeatable. Concluding remarks are given in Section 4. An Appendix is devoted to study the Lyapunov stability for differential equations in
infinite dimensional spaces.

\section{Preliminaries and Problem Description}

Let $S$ be a Polish space (i.e., complete separable metric space) with the associated metric $d$. We consider a symmetric two player game
$G = (S, u)$. Here $u : S \times S \to \mathbb{R}$ represents the payoff function, which is bounded and measurable. Recall
that in a symmetric game if a player  chooses $z \in S$ and the other player chooses  $w \in S$, then the player choosing
$z$ will get a payoff $u(z, w)$.

Let $\mathcal{B}$ denote the Borel sigma-algebra on $S$, i.e., the sigma-algebra generated by all open sets in $S$.
Following the tradition of evolutionary game theory, a population state of the game $G$ is defined to be a probability measure, $Q$, on the measurable space $(S,\mathcal{B})$. The set of all population states is denoted by $\Delta$. The average payoff to a population $P$ playing against a population $Q$ is given by
\[
E(P,Q) := \int_S \int_S u(z,w) ~Q(dw)~P(dz).
\]
We recall few definitions from evolutionary game theory.

\begin{definition}[\cite{MaS,Bomze}]
A population state $P$ is called an \emph{evolutionarily stable strategy} if for every ``mutation'' $Q \neq P$, there is an invasion barrier $\epsilon(Q) > 0$,  such that, for all $0 < \eta \leq \epsilon(Q)$,
\begin{equation}\label{ess condition}
E(P,(1-\eta)P+\eta Q) > E(Q,(1-\eta)P+\eta Q).
\end{equation}
\end{definition}

\begin{definition}[\cite{ViCa}]
A population state $P$ is called \emph{uninvadable} if, in Definition 1, $\epsilon(Q)$ can be chosen independent of $Q \in \Delta$, $Q \not = P$.
\end{definition}

Note that, we can rewrite the condition \eqref{ess condition} in the ESS definition as
\[
E(P,R) > E(R,R),
\]
where $R = (1-\eta)P+\eta Q$ for all $0 < \eta \leq \epsilon(Q)$. A neighbourhood of $P$ can be completely characterized by $R$, with $\eta$ sufficiently small, when the set of pure strategies is finite; but not when the set of pure strategies is infinite. In games with infinite strategy set, the neighbourhoods of the population state $P$ can be determined using various topologies. In this article we consider the topology generated by the variational (or strong) norm i.e., the variational (or strong) topology. The variational norm of a probability measure $P$
is given by
\[
\| P \|  = 2~ \sup_{B \in \mathcal{B}} |P(B)|.
\]
Thus the distance between two probability measures $P$ and $Q$ is given by
\[
\| P-Q \|  = 2~ \sup_{B \in \mathcal{B}} |P(B)-Q(B)|.
\]
We next define strong uninvadability and strong unbeatability.

\begin{definition}[\cite{Bom,Bom1}]
A population state $P$ is called \emph{strongly uninvadable} if there is an $\epsilon > 0$ such that for all
population states $R \neq P$ with $\| R - P \| \leq \epsilon$, we have
\[
E(P,R) > E(R,R).
\]
\end{definition}

\begin{definition}[\cite{BomzeWeibull}]
A population state $P$ is called \emph{strongly unbeatable} if there is an $\epsilon > 0$ such that for all population states $R \neq P$ with $\|R-P\| \leq \epsilon$, we have
\[
E(P,R) \geq E(R,R).
\]
\end{definition}

It can be easily seen that a strongly uninvadable state is uninvadable and an uninvadable state is an ESS (\cite{Bom1}).

We now consider the evolution of the population over time using the replicator dynamics (\cite{OR1,OR2}). To this end, we note that the success (or lack of success) of a strategy $z \in S$ against a strategy $w \in S$ is given by
\[
\sigma(z,w) := u(z,w) - u(w,w).
\]
The average success (or lack of success) of a strategy $z \in S$ against a population $Q \in \Delta$ is given by
\[
\sigma(z,Q) := \int_S u(z,w) ~ Q(dw) - \int_S \int_S u(\bar{z},\bar{w})~ Q(d\bar{w})~ Q(d\bar{z}) = E(\delta_z,Q) - E(Q,Q),
\]
where the Dirac measure $\delta_z$ represents a monomorphic population state.

The replicator dynamics is derived based on the idea that the relative increment in the frequency of strategies in a set $B\in \mathcal{B}$ is given by the average success of strategies in $B$. That is, for every $B \in \mathcal{B}$,
\begin{equation}\label{re}
 Q'(t)(B) = \frac{dQ(t)}{dt}(B) =  \int_B \sigma(z,Q(t))~Q(t)(dz)
\end{equation}
where $Q(t)$ denotes the population state at time $t$.

The replicator dynamics equation \eqref{re} can be also written as
\begin{equation}\label{red}
Q'(t) = F(Q(t)),
\end{equation}
where for every $B \in \mathcal{B}$, $F(Q(t))(B) = \int_B \sigma(z,Q(t))~Q(t)(dz)$; that is, $F(Q(t))$ is the signed measure whose Radon-Nikodym derivative $\dfrac{dF(Q(t))}{dQ(t)}$, w.r.t.\ $Q(t)$ is $\sigma(\cdot,Q(t))$.

Since the payoff function $u$ is bounded and measurable, it follows that the replicator dynamics is well posed (\cite[Theorem 2]{OR1}) which in turn assures the existence of a unique solution to the replicator dynamics \eqref{re} with the initial condition $Q(0)$.

We can now introduce a few dynamic stability definitions for population states.
Let $P$ be a rest point of the replicator dynamics, i.e., $F(P) = 0.$
\begin{definition}
Rest point $P$ is called \emph{Lyapunov stable} if for all $\epsilon > 0$, there exists an $\eta > 0$ such that,
\[
||Q(0)-P|| < \eta ~~ \Rightarrow ~ ||Q(t)-P|| < \epsilon \mbox{~~~for all~} t > 0.
\]
\end{definition}

\begin{definition}
$P$ is called \emph{strongly attracting} if there exists an $\eta > 0$ such that $Q(t)$ converges to $P$ strongly as $t \to \infty$, whenever $||Q(0) - P|| < \eta$.
\end{definition}

\begin{definition}
$P$ is called \emph{asymptotically stable} if $P$ is Lyapunov stable and strongly attracting.
\end{definition}

One of the main interest in studying games with continuous strategy spaces is to establish conditions under which the population states will be dynamically stable. We recall here some of the existing results in this direction.

\cite{OR1} provide with sufficient conditions for a monomorphic population state $Q^* = \delta_x$ to be Lyapunov stable and ``weakly attracting". More precisely, they prove the following result.

\begin{theorem}[\cite{OR1}]
If $Q^*=\delta_x$ is an uninvadable, monomorphic population state, then $Q^*$ is Lyapunov stable. Moreover, if $u$ is continuous then $Q^*$ is weakly attracting, in the sense that the trajectory w.r.t the replicator dynamics converges to $Q^*$ weakly when the initial population state is from a small (strong) neighbourhood of $Q^*$.
\end{theorem}

In \cite{Bom}, there are a couple of results regarding the asymptotic stability of population states under very strong assumptions. The first theorem that he gives is as follows.
\begin{theorem}[\cite{Bom}] \label{Th2 bomze}
Suppose that $\Delta$ is relatively $\tau$-compact, where $\tau$ is a topology on the L-space, $\mathcal{L}$, of $(S, \mathcal{B}, \Delta)$ such that the map $Q \mapsto \|Q\|$ form $\mathcal{L}$ to $\mathbb{R}$ is lower semicontinuous. If $P \in \Delta$ is strongly uninvadable, and if the map $Q \mapsto E(P,Q)-E(Q,Q)$ on $\Delta$ is $\tau$-continuous, then every replicator dynamics trajectory $Q(t)$, $t\geq 0$, starting in
\[
\mathcal{U}_P = \left \{ Q \in \Delta : P \ll Q~ \mbox{and}~ \int_S \ln \left ( \frac{dP}{dQ} \right ) ~dP < \delta \right \}
\]
converges to $P$ as $t \to \infty$, with respect to $\tau$, provided that $\delta > 0$ is small enough.
\end{theorem}

The $\tau$-compactness condition in the above theorem is a very strong condition when $\tau$ is taken to be the strong topology. Bomze gives another result regarding the asymptotic stability of the population state $P$ with $\tau$ as the strong topology under the condition of the following theorem.

Let $\mathcal{M}$ be the linear span of $\Delta$, with variational norm, and $\mathcal{F}$ be the space of all bounded measurable functions with the norm $\|F\|_\infty = \sup_{z\in S}|F(z)|$. Also, for $F_Q \in \mathcal{F}$, $F_Q(z)$ denotes the mean payoff to $z \in S$ against $Q \in \Delta$.

\begin{theorem}[\cite{Bom}]\label{Th1 bomze}
Let $P \in \Delta$ be a rest point and assume that the map $Q \mapsto F_Q$ from $\Delta$ to $\mathcal{F}$ is Fr\'{e}chet differentiable at $Q=P$ in the sense that there is a continuous linear map $DF_P : \mathcal{M} \to \mathcal{F}$ such that for all $\eta > 0$ there is a $\rho > 0 $ fulfilling
\[
\|F_Q - F_P - DF_P(Q - P)\|_\infty \leq \eta \|Q-P\|
\]
\[
\mbox{whenever}~~~\|Q-P\| < \rho~~\mbox{and}~ Q \in \Delta.
\]
$P$ is strongly uninvadable if there is a constant $c > 0$ such that
\begin{equation}\label{neggame}
\int_S DF_P(Q-P)~d(Q-P) \leq -c ~\|Q-P\|^2~~\mbox{for all }Q\in \Delta.
\end{equation}
\end{theorem}

Note that games satisfying this last condition are known as negative definite games. Negative definite games possess many interesting properties and they have been studied extensively in the literature (\cite{Sandholm,Cheung2014,LahkarRiedel2015}).

The above theorem gives a first-order condition for $P$ to be strongly uninvadable. The next one gives another set of conditions for asymptotic stability of a population state $P$.

\begin{theorem}[\cite{Bom}]\label{Th3 bomze}
Under the assumptions of Theorem \ref{Th1 bomze}, every replicator dynamics trajectory $Q(t)$, $t\geq 0$, starting in $\mathcal{U}_P$ (as in Theorem \ref{Th2 bomze}) satisfies $\|Q(t) - P\| \to 0$ as $t \to \infty$.
\end{theorem}

The first-order condition for strong uninvadability of $P$, given in Theorem \ref{Th1 bomze} is not a necessary condition, as illustrated by Example \ref{counter example} in the next section. We observe that the conditions for stability can be weakened when we are dealing with polymorphic population states. In the next section we focus on the polymorphic population states and provide with conditions for their stability with the underlying topology as the strong topology.

\section{Stability of Polymorphic Population States}

In this section, we will first study some properties of polymorphic population states and then we will move on to the stability of these population states. As the name suggests, polymorphic population states have a finite support. Moreover we can view them as convex combinations of monomorphic population states.

\subsection{Properties of Polymorphic Population States}

We begin by characterizing rest points of the replicator dynamics \ref{red}.

\begin{lemma}\label{RPLG}
A population state $P$ is a rest point of the replicator dynamics \ref{red} if and only if $\int_S u(z, w) P(dw)$ is constant a.s. $z(P)$.
\end{lemma}

\begin{proof}
Clearly, $P$ is a rest point of the replicator dynamics if and only if for all $B \in \mathcal{B}$,
\[
F(P)(B) =\int_B \sigma(z,P)~ P(dz) =  0.
\]
This is equivalent to
\[
\sigma(\cdot, P) = 0 ~~\mbox{a.s.}(P).
\]
This implies and is implied
by
\[
 E(\delta_{z},P) = E(P,P) ~~~~~\mbox{a.s. } z(P).
\]
From this it follows that $P$ is a rest point of the replicator dynamics if and only if,  $ \int_S u(z, w) P(dw)$  is independent of $z$ a.s.$(P)$.
\qed
\end{proof}

In the case of the polymorphic population state given by
\begin{equation}\label{P}
P^* = \alpha_1 \delta_{x_1} + \alpha_2 \delta_{x_2} + \cdots + \alpha_k \delta_{x_k},
\end{equation}
where $x_1, x_2, \cdots, x_k$ are distinct points in $S$ and the sum of the positive numbers $\alpha_1, \alpha_2, \cdots, \alpha_k$ is $1$, this lemma reduces to the following corollary.

\begin{corollary}\label{rest point lemma}
Let $P^*$ be a polymorphic population state given by \eqref{P}. Then, $P^*$ is a rest point of the replicator dynamics if and only if the sum $  \sum \limits_{j=1}^k \alpha_j u(x_i,x_j)$ is independent of $i$.
\end{corollary}

\begin{proof}
The proof follows since the support of  $P^*$ is  $\{x_1, x_2, \cdots, x_k \}$.
\qed
\end{proof}

We illustrate the above corollary using the following example.

\begin{example}
Let $S = [0,1]$ and the payoff function be defined by
\[
u(z,w) = \left \{ \begin{array}{cl}
                    w & \text{if } z < w\\
                    z-w & \text{if } z \geq w
                  \end{array}
\right.
\]
Consider the polymorphic population state
\[
P^*  = \alpha_1 \delta_{x_1} + \alpha_2 \delta_{x_2} + \alpha_3 \delta_{x_3} = \frac{1}{3} \delta_0 + \frac{1}{3} \delta_{1/2} + \frac{1}{3} \delta_1.
\]
Then for $i=1$,
\begin{align*}
\sum_{j=1}^k \alpha_j u(x_1,x_j) & = \frac{1}{3} \sum_{j=1}^3 u(0,x_j)  = \frac{1}{3} \left \{ u(0,0) + u(0,1/2) + u(0,1) \right \} \\[2mm]
& = \frac{1}{3} \left \{ 0 + \frac{1}{2} + 1 \right \}  = \frac{1}{2}.
\end{align*}
Similarly, for $i=2$ and $i=3$, we get the sum $\sum_{j=1}^k \alpha_j u(x_i,x_j)$ as $1/2$. Thus, by Corollary \ref{rest point lemma}, $P^*$ is a rest point of the replicator dynamics.
\end{example}

Now that we have established the condition for the polymorphic population state $P^*$ to be a rest point of the replicator dynamics, we move on to characterizing small neighbourhoods of $P^*$ with respect to the variational topology.

Consider population states $P$ and $Q$ from $\Delta$. Then by Lebesgue decomposition, we can decompose $Q$ in terms of the Borel measures $Q_1$ and $Q_2$ such that
\[
Q =  Q_1 +  Q_2
\]
where $Q_1$ is absolutely continuous with respect to $P$ and $Q_2$ is singular with respect to $P$.  Now
\[
\|Q - P \| = 2 \sup_{B \in \mathcal{B}} | Q_1(B) + Q_2(B) - P(B) | \geq 2 |Q_1 (A) - P(A)|
\]
for every Borel set  $A \subseteq \supp{P}$.

We can similarly decompose the population state $Q$, by taking $P = P^*$, the polymorphic population state. Now since the support of $P^* = \{x_1, x_2, \cdots, x_k\}$, note that whenever the support of $Q$ is a strict subset of the support of $P^*$, form the above inequality, we obtain
\[
\|Q - P^* \|  \geq  2~ \inf ~ \{ \alpha_j : x_j \not \in \supp{Q} \}.
\]
Let $ 0 < \epsilon <  2~\inf ~\{ \alpha_j : j = 1, 2, \cdots, k \}$.
From the above, it follows that every $Q$ in the $\epsilon$-neighbourhood of $P^*$, the support of $Q_1$ must be equal to
the support of $P^*$.\\

In conclusion, every  population state $Q$ sufficiently close to $P^*$ will be of the form
\begin{equation}\label{nbhd population}
Q = \sum_{j=1}^k \beta_j \delta_{x_j} + \beta_{k+1} R~; ~~~~\sum_{j=1}^{k+1} \beta_j = 1
\end{equation}
where the support of $R \in \Delta$ is disjoint from the support of $P^*$.

One useful consequence of this fact is  the following lemma, whose proof is a straight forward application
of the representation \eqref{nbhd population}.

\begin{lemma}\label{abscont}
Let $P^*$ be the polymorphic state given by \eqref{P}. Then for sufficiently small $\epsilon$, $P^*$ is absolutely continuous with respect to $Q$, for every $Q$ in $\epsilon$-neighbourhood of $P^*$.
\end{lemma}

Not surprisingly, this lemma fails when $P^*$ is infinitely supported. In fact, the Lebesgue measure on $[0, 1]$ provides a counter example.\\

Another consequence of the equation \eqref{nbhd population} is the following lemma which gives bounds for the variational distance between population states in a small neighbourhood of $P^*$.

\begin{lemma}\label{bounds}
Let $\epsilon > 0$ be small enough such that all population states in the neighbourhood $\Omega(\epsilon) := \{Q \in \Delta : \|Q-P^*\| < \epsilon\}$ are of the form \eqref{nbhd population}. If
\begin{align*}
Q_1 = \sum_{j=1}^k \beta_j \delta_{x_j} + \beta_{k+1} R_1~; ~\sum_{j=1}^{k+1} \beta_j = 1,\\[2mm]
Q_2 = \sum_{j=1}^k \gamma_j \delta_{x_j} + \gamma_{k+1} R_2~; ~\sum_{j=1}^{k+1} \gamma_j = 1
\end{align*}
are population states in $\Omega(\epsilon)$ then we have,
\[
2 \max_{1 \leq j \leq k} \left \{ |\beta_j - \gamma_j| \right \} \leq \|Q_1-Q_2\| \leq 2 \max \left \{ \sum_{j=1}^k |\beta_j - \gamma_j|~, ~2 \left ( 1 - \sum_{j=1}^k \beta_j \right ) \right \}.
\]
\end{lemma}

The proof of this lemma is omitted as it follows from straight forward calculations of the variational distance. From this lemma, we can write the bounds for variational distance of $P^*$ and a population state $Q$ in its neighbourhood, with the form given in \eqref{nbhd population}, as
\begin{equation}\label{bounds wrt P^*}
\max_{1 \leq j \leq k} |\alpha_j - \beta_j | \leq \frac{1}{2}\|Q - P^*\| \leq \ \max \left\{ \sum_{j=1}^k |\alpha_j - \beta_j|, 2 \left(1 - \sum_{j=1}^k \beta_j \right ) \right\}.
\end{equation}

\begin{remark}\label{convergence of weights}
The above lemma and its application not only gives us lower and upper bounds for the variational distance but it also proves that $\|Q-P^*\| \to 0$ if and only if $|\beta_j - \alpha_j| \to 0$ for every $j = 1, 2, \cdots, k$. Thus to prove convergence of a population state to $P^*$, it is enough to prove the convergence of the weights on each of $x_j$'s.
\end{remark}

Before proceeding further to study the stability of polymorphic population states, we present with an example which shows that the first order condition in Theorem \ref{Th1 bomze} is not necessary to guarantee that $P^*$ is strongly uninvadable.

\begin{example}\label{counter example}
Let $S = [-1,1]$ and the payoff function be defined as
\[
u(z,w) = 2-zw~~\mbox{for all }z,w \in S.
\]
The polymorphic  state $P^* = \alpha \delta_{-1} + (1 - \alpha) \delta_1$ with $\alpha = 1/2$ is a rest point of the replicator dynamics. Now consider a population state $Q$ from an arbitrarily small strong neighbourhood of $P^*$ (as given in Lemma \ref{abscont}). Then, $Q$ will be of the form
\[
Q = \beta \delta_{-1} + \gamma \delta_1 + (1-\beta-\gamma)R,
\]
where $R\in \Delta$ such that $R(\{-1, 1\}) = 0$ and $0< \beta+\gamma \leq 1$.

Note that $E(\delta_z,P^*) = 2 $ for all $z \in S$ which implies that
\[
E(P^*,P^*) = E(P^*,Q) = E(Q,P^*) = 2.
\]
By definition of $u$,
\[
E(Q,Q) = 2 - \left (\gamma - \beta + (1-\beta-\gamma) \mu \right )^2
\]
where $\mu = \int_S z R(dz)$.

Therefore, $P^*$ is strongly uninvadable since $E(P^*,Q) - E(Q,Q) > 0$ for every $Q\not = P^* $ in a
strong neighbourhood of $P^*$.

However, we can show that the condition \eqref{neggame} is not true. In fact, the map $F_Q = \int_S u(\cdot,w)~Q(dw)$
is Fr\`{e}chet differentiable and $DF_{P^*}(Q-P^*) = F_{Q-P^*}$. Hence,
\[
\int_S DF_{P^*}(Q-P^*)~d(Q-P^*) = E(Q,Q) - E(P^*,Q) - E(Q,P^*) + E(P^*,P^*).
\]
Taking $Q = \frac{1}{2} (\delta_{-1/2} + \delta_{1/2})$ we note that
\[
\int_S DF_{P^*}(Q-P^*)~d(Q-P^*) = 0.
\]
Thus the game is not negative definite game.
\end{example}

\subsection{Stability of $P^*$}

We are now ready to discuss the stability of polymorphic population states. First, we recall the following result
from \cite[Proposition 13]{VeelenSpreij}. Since this result forms the background for our work, we provide a proof which is slightly different from that of van Veelen and Spreij.

\begin{theorem}[\cite{VeelenSpreij}]
Every asymptotically stable rest point of the replicator dynamics in variational distance is finitely supported.
\end{theorem}

\begin{proof}
From \cite[Lemma 2]{Bom1}, we have
\[
\supp{Q(t)} = \supp{Q}
\]
where $Q(\cdot)$ is the trajectory of the replicator dynamics \eqref{re} with initial condition $Q(0) = Q$.
If  $Q(t)$ converges to $P$ strongly, then, by  Portmanteau theorem \cite[Theorem 2.1]{BillingsleyConv}, we must have
\begin{equation}\label{support}
\supp{P} \subseteq \supp{Q}.
\end{equation}
To prove the theorem, we exhibit a probability measure $Q$ in any arbitrary neighbourhood of $P$ contradicting \eqref{support},
provided $P$ is not finitely supported.

If $P$ is not finitely supported, then for each $\epsilon > 0$, we can
find a set $C$ such that $0 < P(C) < \epsilon$ \cite[Lemma 15]{VeelenSpreij}. Choose $Q$ which is defined by
\[
Q(B) = \frac{1}{ 1 - P(C) } Q(B \setminus C), ~~\mbox{for}~ B \in \mathcal{B}.
\]
Now, it is easy to verify that $\| P - Q \| < \epsilon$, giving the required contradiction.
\qed
\end{proof}

\begin{remark}
In fact, the above proof  also proves the result in the case of weak convergence. See  Proposition 14 in \cite{VeelenSpreij}.
\end{remark}

Let $P^*$ be a rest point of the replicator dynamics where $P^*$ is as in \eqref{P}.
Let $Q(0)$  be a population state in a small neighborhood of  $P^*$ as in Lemma \ref{abscont}.  Hence
\begin{equation}\label{initial state}
Q(0) = \sum_{j=1}^k \beta_j \delta_{x_j} + \beta_{k+1} R(0)~; ~~~~\sum_{j=1}^{k+1} \beta_j = 1
\end{equation}
where $R(0) \in \Delta$ with $R(0)(\{x_1, x_2, \cdots, x_k\}) = 0$.

Consider the solution $Q(\cdot)$  of the replicator dynamics equation \eqref{re} starting from $Q(0)$.
Since the support of $Q(0)$ and  $Q(t)$ is the same, $Q(t)(\{x_j\}) > 0$ for all $j = 1,2, \cdots, k$.

Using this, from the replicator dynamics equation \eqref{re}, we  obtain,
\begin{equation}\label{replfinite}
Q^\prime(t)(\{x_j\}) = Q(t)(\{x_j\}) ~\sigma(x_j, Q(t)), ~~~~ Q(0)(\{x_j\}) = \beta_j
\end{equation}
for $j=1, 2, \cdots, k$.

We are, now, ready to prove the stability of polymorphic population state with the following theorem which establishes its Lyapunov stability.

\begin{theorem}\label{Lyapunov Stability of P^*}
Let $P^*$ be the polymorphic population state as in \eqref{P}. If $P^*$ is strongly unbeatable then $P^*$ is Lyapunov stable.
\end{theorem}

\begin{proof}
 Let the polymorphic population state $P^*$ be strongly unbeatable. Then, there exists $\epsilon > 0$ such that for $R (\not = P^*)$, with $\|R-P^*\| \leq \epsilon$,
 \[
 E(P^*,R) \geq E(R,R)
 \]
 Let $\delta < 2 \min \{\alpha_1, \alpha_2, \cdots, \alpha_k\}$, $\theta = \min \{\epsilon, \delta\}$ and $\Omega = \{Q \in \Delta : \|Q-P^*\| < \theta\}$.

 By the definition of $\theta$, it follows from Lemma \ref{abscont} that $P^*$ is absolutely continuous with respect to $Q$, for every $Q \in \Omega$. Therefore, for every $Q \in \Omega$ and $B \in \mathcal{B}$, we have
 \[
    P^*(B) = \int_B \frac{dP^*}{dQ}~dQ.
 \]
 Putting $B = \{x_j\}$, we get,
 \begin{equation}\label{density values}
 \frac{dP^*}{dQ}(x_j) = \frac{\alpha_j}{Q(\{x_j\})};~~~~j=1,2,\cdots,k.
 \end{equation}
 Define $V:\Omega \to \mathbb{R}$ by,
 \begin{equation}\label{lyapunov fn}
   V(Q) = \int_S \ln \left ( \frac{dP^*}{dQ} \right ) ~dP^*.
 \end{equation}
 Since $P^*$ is polymorphic, using \eqref{density values}, we can rewrite \eqref{lyapunov fn} as follows.
 \begin{equation}\label{lyapunov fn 1}
    V(Q) = \sum_{j=1}^k \alpha_j ~ \ln \left ( \frac{dP^*}{dQ}(x_j) \right ) = \sum_{j=1}^k \alpha_j ~ \ln \left ( \frac{\alpha_j}{Q(\{x_j\})} \right ).
 \end{equation}
 Using continuity of the log function, one may show that $V$ is continuous in $\Omega$. Moreover, $V(P^*) = 0$ and for $Q \in \Omega$ such that $Q \not = P^*$, we have,
    \begin{align*}
        V(Q) & = \sum_{j=1}^k \alpha_j \ln \left ( \frac{\alpha_j}{Q(\{x_j\})} \right )\\
        & = - \sum_{j=1}^k \alpha_j \ln \left ( \frac{Q(\{x_j\})}{\alpha_j} \right )\\
        & > - \sum_{j=1}^k \alpha_j \left ( \frac{Q(\{x_j\})}{\alpha_j} - 1 \right )~~~~(\because ln(z) < z - 1 ~~\mbox{for}~~z \not = 1)\\
        & = 1 - \sum_{j=1}^k Q(\{x_j\})~~~\geq 0.
    \end{align*}
 Thus, $V(Q) \geq 0$ and the equality holds if and only if $Q = P^*$; in other words, $V$ is positive definite.

 From Pinsker's inequality ( see \cite[(3.3.6) and (3.3.9)]{Reiss} and  \cite[Lemma 3] {Bom1}) we see that $\|Q-P^*\|^2 \leq V(Q)$ for every $Q \in \Omega$.

 Now, let $Q(t)$ be the trajectory of the replicator dynamics with the initial population state as $Q \in \Omega$. Then,
    \begin{align*}
        \frac{d}{dt}V(Q(t)) & = \frac{d}{dt} \left ( \sum_{j=1}^k \alpha_j~\ln \left (\frac{\alpha_j}{Q(t)(\{x_j\})} \right ) \right ) ~~~~~ \mbox{(from \eqref{lyapunov fn 1})}\\
        & = - \sum_{j=1}^k \alpha_j \frac{d}{dt} \left ( \ln \left( \frac{Q(t)(\{x_j\})}{\alpha_j} \right) \right )\\
        & = - \sum_{j=1}^k \alpha_j~ \frac{Q'(t)(\{x_j\})}{Q(t)(\{x_j\})}\\
        & = - \sum_{j=1}^k \alpha_j~ \sigma(x_j, Q(t))~~~~~~\mbox{(from \eqref{replfinite})}\\
        & = - \sum_{j=1}^k \alpha_j~ [E(\delta_{x_j},Q(t)) - E(Q(t),Q(t))]\\
        & = - E(P^*,Q(t)) + E(Q(t),Q(t)).
    \end{align*}
Therefore,
    \[
    \dot{V}(Q) = - E(P^*,Q) + E(Q,Q).
    \]
    Since, $P^*$ is strongly unbeatable, $\dot{V}(Q) \leq 0$ for any $Q \in \Omega$ which proves that $V$ is non-increasing along replicator dynamics trajectories.

    Thus, by Theorem \ref{Lyapunov Stability} (in the Appendix), we can conclude that $P^*$ is Lyapunov stable.
\qed
\end{proof}

The above theorem establishes the Lyapunov stability of unbeatable polymorphic population states. We next prove a result regarding their asymptotic stability.

\begin{theorem}\label{Asymptotic Stability of P^*}
Let $P^*$ be the polymorphic population state as in \eqref{P}. If $P^*$ is strongly uninvadable then $P^*$ is asymptotically stable.
\end{theorem}

\begin{proof}

Let the polymorphic population state $P^*$ be strongly uninvadable. Then there exists $\epsilon > 0$ such that for all $R (\not = P^*)$ with $\|R-P^*\| \leq \epsilon$,
\[
E(P^*,R) > E(R,R).
\]
We can define $\Omega$ and the function $V$ as in the proof of Theorem \ref{Lyapunov Stability of P^*} where $V$ is a positive definite continuous function for which
\[
    \dot{V}(Q) = - E(P^*,Q) + E(Q,Q)
\]
for every $Q \in \Omega$. Since, $P^*$ is strongly uninvadable, $\dot{V}(Q) < 0$ for any $Q \in \Omega$, $Q \not = P^*$ which proves that $V$ is strictly decreasing along replicator dynamics trajectories which remain in $\Omega$.\\ \\*
Now, for any $0 < \epsilon_1 < \theta$, by Theorem \ref{Lyapunov Stability of P^*}, there exists $\delta_1 > 0$ such that every trajectory starting from the open ball centered at $P^*$ with radius $\delta_1$ (denoted by $B(P^*,\delta_1)$), will remain in $B(P^*,\frac{\epsilon_1}{2k})$.\\ \\*
Consider the trajectory $Q(t) = Q(t;Q_0)$ starting from $Q_0 \in B(P^*,\delta_1)$. For this trajectory $Q(t)$, clearly, there exists a sequence $t_n \to \infty$ such that $Q(t_n)(\{x_j\})$ converges to a limit, say $\beta_j^*$; $j = 1,2,\cdots,k$.\\ \\*
Since $Q(t_n) \in B(P^*,\frac{\epsilon_1}{2k})$, it follows from \eqref{bounds wrt P^*} that $|\alpha_j - \beta_j^*| \leq \frac{\epsilon_1}{2k}$ for every $j = 1,2,\cdots,k$ and hence $\sum \limits_{j=1}^k |\alpha_j - \beta_j^*| \leq \frac{\epsilon_1}{2} < \theta$. In particular, by the definition of $\theta$ we now have $\beta_j^* > 0$ for every $j = 1,2,\cdots,k$.

This implies that
\[ V(Q(t_n))  = \sum_{j=1}^k \alpha_j \ln \left ( \frac{\alpha_j}{Q(t_n)(\{x_j\})} \right )
\]
converges to
\[ V(Q^*) = \sum_{j=1}^k \alpha_j \ln \left ( \frac{\alpha_j}{\beta_j^*} \right )
\]
for any (fixed) $Q^* \in \Lambda \subset \Omega$ where
\[
\Lambda = \left \{Q \in \Omega ~\mid~ Q = \sum_{j=1}^k \beta_j^* \delta_{x_j} + \left ( 1 - \sum_{j=1}^k \beta_j^* \right ) R; ~~R(\{x_1,x_2,\cdots,x_k\}) = 0 \right \}.
\]
For $s > 0$, by the replicator dynamics equations \eqref{replfinite}, we know that
\begin{align}\label{initial at Q^*}
Q(s;Q^*)(\{x_j\}) & = \beta_j^* ~ \exp\left(\int_0^s \sigma\left(x_j,Q(t;Q^*) \right)~dt \right) \nonumber\\
& = \beta_j^* ~T(s)
\end{align}
and
\begin{align}\label{initial at Q(t_n)}
Q(s;Q(t_n))(\{x_j\}) & = Q(t_n)(\{x_j\}) ~ \exp\left(\int_0^s \sigma\left(x_j,Q(t;Q(t_n)) \right)~dt \right) \nonumber\\
& = Q(t_n)(\{x_j\}) ~T_n(s).
\end{align}
Therefore we have,
\begin{align}
& | V(Q(s;Q^*)) - V(Q(s,Q(t_n))) |\nonumber\\[2mm]
 = &\Bigg | \sum_{j=1}^k \alpha_j \ln \left ( \frac{\alpha_j}{\beta_j^* ~T(s)}\right ) - \sum_{j=1}^k \alpha_j \ln \left ( \frac{\alpha_j}{Q(t_n)(\{x_j\}) ~T_n(s)}\right ) \Bigg | \nonumber\\[2mm]
 = & \sum_{j=1}^k \alpha_j \Bigg | \ln \left ( \frac{Q(t_n)(\{x_j\}) ~T_n(s)}{\beta_j^* ~T(s)} \right ) \Bigg | \label{V diff form}
\end{align}
Since $\sigma(\cdot,Q)$ is bounded, it follows that $\dfrac{T_n(s)}{T(s)} \to 1$ (uniformly in $n$) as $s \downarrow 0$ and hence from \eqref{V diff form}, we get,
\[
\lim_{s \downarrow 0 ,~ n \uparrow \infty} | V(Q(s;Q^*)) - V(Q(s,Q(t_n))) | = 0.
\]
Thus, by Theorem \ref{Asymptotic Stability} (in Appendix), we can conclude that $P^*$ is asymptotically stable.
\qed
\end{proof}

In the next section we make some concluding remarks which is followed by Appendix that gives complete proof of general Lyapunov stability results.

\section{Conclusions}

In this article, we studied the stability of polymorphic population states in games with continuous strategy spaces. We proved that strong uninvadability is a sufficient condition for asymptotic stability of a polymorphic population state whereas, strong unbeatability is enough for the Lyapunov stability. Beyond finitely supported population states, one cannot establish similar stability results unless we weaken the notion of stability. This is an interesting future research topic in games with continuous strategy spaces.

\section*{Appendix}

Here we establish two abstract stability theorems used to prove our main theorems in Section 3. To this end, we consider an abstract differential equation
\begin{equation}\label{Differential Equation}\tag{A.1}
\phi'(t) = H(\phi(t))
\end{equation}
on a Banach space $(X,\|\cdot\|_X)$. It is assumed that for each initial condition $\phi_0$ in an invariant set $Y \subset X$, the differential equation \eqref{Differential Equation} has a unique solution $\phi(t) = \phi(t;\phi_0)$ defined for every $t \geq 0$. We want to analyze this system around a rest point $\phi^* \in Y$. We recall the definition of $\mathcal{K}_0^\infty$ functions:
\begin{align*}
\mathcal{K}_0^\infty  & = \{\omega : [0,\infty) \to [0,\infty) ~\mid ~\omega \mbox{ is strictly increasing, continuous, }\\
& ~~~~~~~~~~~~~\omega(0) = 0 \mbox{ and } \lim_{s \to \infty} \omega(s) = \infty \}.
\end{align*}

\begin{theorem}\label{Lyapunov Stability}
Let $\Omega$ be an open subset of $Y$ containing the rest point $\phi^*$ of \eqref{Differential Equation}. Assume that $V : \Omega \to \mathbb{R}$  is continuous at $\phi^*$ and satisfies
\begin{itemize}
\item[(i)] $V(\phi) \geq 0$ on $\Omega$ and $V(\phi^*) = 0$;
\item[(ii)] there exists $\omega \in \mathcal{K}_0^\infty$ such that $w(\|\phi - \phi^*\|_X) \leq V(\phi)$ for all $\phi \in \Omega$;
\item[(iii)] $V$ is non increasing along trajectories of \eqref{Differential Equation} that lie in $\Omega$.
\end{itemize}
Then $\phi^*$ is Lyapunov stable.
\end{theorem}

\begin{proof}
Let $B(\phi^*,\epsilon)$ be the open ball around $\phi^*$ with radius $\epsilon$. Let $\epsilon > 0$ be small enough such that the closure of $B(\phi^*,\epsilon)$ is contained in $\Omega$.

By continuity of $V$ at $\phi^*$ and condition $(i)$, there exists $\delta > 0$ such that $V(\phi) < \omega(\epsilon)$ whenever, $\phi \in B(\phi^*,\delta)$.

Clearly, by condition $(iii)$ the set $U = \{ \phi \in \Omega ~|~ V(\phi) < \omega(\epsilon) \}$ is invariant. Without loss of generality we can assume that closure of $U$ is a subset of $\Omega$. Therefore, for every $\phi_0 \in B(\phi^*,\delta)$, the trajectory $\phi(t) = \phi(t;\phi_0)$ lies in $U$ and hence,
\[
\omega \left (\|\phi(t) - \phi^*\|_X \right ) \leq V(\phi(t)) < \omega(\epsilon).
\]
As $\omega \in \mathcal{K}_0^\infty$, it is invertible and from above it follows that
\[
\|\phi(t) - \phi^*\|_X < \epsilon.
\]
Thus we have proved that the trajectory $\phi(t)$ lies in $B(\phi^*,\epsilon)$ whenever $\phi_0 \in B(\phi^*,\delta)$.
\qed
\end{proof}

\begin{theorem}\label{Asymptotic Stability}
Let $\Omega$ be an open subset of $Y$ containing the rest point $\phi^*$ of \eqref{Differential Equation}. Assume that $V : \Omega \to \mathbb{R}$  is continuous on $\Omega$ and satisfies
\begin{itemize}
\item[(i)] $V(\phi) \geq 0$ on $\Omega$ and $V(\phi^*) = 0$;
\item[(ii)] there exists $\omega \in \mathcal{K}_0^\infty$ such that $w(\|\phi - \phi^*\|_X) \leq V(\phi)$ for all $\phi \in \Omega$;
\item[(iii)] $V$ is strictly decreasing along trajectories of \eqref{Differential Equation} that lie in $\Omega \setminus \{\phi^*\}$;
\item[(iv)] there exists $\delta_1 > 0$ such that for every trajectory $\phi(t)$ emanating from $B(\phi^*,\delta_1)$, there exists a sequence $t_n \to \infty$ such that $V(\phi(t_n))$ converges to $V(\psi)$ for some $\psi \in \Omega$ and
    \[\lim_{s \downarrow 0 ,~ n \uparrow \infty} |V(\phi(s;\psi)) - V(\phi(s,\phi(t_n)))| = 0.
    \]
\end{itemize}
Then $\phi^*$ is asymptotically stable.
\end{theorem}

\begin{proof}
As the Lyapunov stability follows from the above theorem it remains to show that $\phi^*$ is attracting.

Let $B(\phi^*,\delta)$ and $B(\phi^*,\epsilon)$ be as defined in the proof of the above theorem. Without loss of generality, we may assume that $\delta \leq \epsilon$. Similarly, there exists $\delta_2 > 0$ such that all trajectories emanating from $B(\phi^*,\delta_2)$ lie in $B(\phi^*,\frac{\delta}{2})$.\\ \\*
Let $\bar{\delta} = \min\{\delta_1,\delta_2\}$ and $\phi(t) = \phi(t;\phi_0)$ be the trajectory of the differential equation \eqref{Differential Equation} with the initial condition $\phi_0 \in B(\phi^*,\bar{\delta})$. Then, by condition $(iv)$, there exists a sequence $t_n \to \infty$ such that $V(\phi(t_n))$ converges to $V(\psi)$ for some $\psi \in \Omega$.\\ \\*
We need to show that $\psi = \phi^*$. By condition $(iii)$, $V(\phi(t)) > V(\psi)$ for every $t \geq 0$.

If $\psi \not = \phi^*$, let $\psi(t) = \phi(t;\psi)$. For any $t > 0$, $V(\psi(t)) < V(\psi)$. By condition $(iv)$,
\[\lim_{s \downarrow 0 ,~ n \uparrow \infty} |V(\phi(s;\psi)) - V(\phi(s,\phi(t_n)))| = 0.
    \]
and hence
\[
V(\phi(s,\phi(t_n))) < V(\psi)
\]
for $s > 0$ small enough and $n$ large enough which is a contradiction because $\phi(s,\phi(t_n)) = \phi(s+t_n;\phi_0)$. Hence, $\psi = \phi^*$.
\qed
\end{proof}

{\bf{Acknowledgements}}

The authors acknowledge the financial support of NBHM through the project ``Evolutionary Stability in Games with Continuous Action Spaces".

\section*{References}
\bibliography{References}

\end{document}